\documentclass[12pt]{amsart}

\usepackage{amsmath,amsthm,amssymb,mathrsfs,amsfonts,verbatim,enumitem,color}
 \usepackage{etoolbox} 
\usepackage{bbm}
\usepackage[all,tips]{xy}
\usepackage{graphicx,ifpdf}
\ifpdf
\fi

\newtheorem{thm}{Theorem}[section]
\newtheorem{lem}[thm]{Lemma}
\newtheorem{cor}[thm]{Corollary}
\newtheorem{prop}[thm]{Proposition}

\theoremstyle{definition}

\theoremstyle{remark}
\newtheorem{rem}[thm]{Remark}

\numberwithin{equation}{section}

\newcommand{\Rmnum}[1]{\expandafter\@slowromancap\romannumeral #1@}
\newcommand{\bv}{\mathbf{v}}

\newcommand{\bb}{\mathbf{b}}
\newcommand{\ba}{\mathbf{a}}
\newcommand{\bm}{\mathbf{m}}
\newcommand{\bn}{\mathbf{n}}
\newcommand{\Xn}{X}

\newcommand{\n}{n}
\newcommand{\m}{m}

\newcommand{\bx}{\mathbf{x}}
\newcommand{\by}{\mathbf{y}}

\newcommand {\ignore}[1]  {}

\newcommand{\vre}{\varepsilon}
\newcommand{\bp}{\mathbf{p}}
\newcommand{\bq}{\mathbf{q}}

\newcommand{\R}{\mathbb{R}}
\newcommand{\N}{\mathbb{N}}
\newcommand{\SL}{\operatorname{SL}}

\newcommand{\Z}{\mathbb{Z}}
\newcommand{\Q}{\mathbb{Q}}

\newcommand{\spa}{\mathrm{span}}

\newcommand{\sm}{\smallsetminus}

\newcommand{\dd}{\, \mathrm{d}}

\newcommand{\ssm}{\smallsetminus}
\newcommand{\nz}{\smallsetminus\{{\mathbf 0}\}}

\newcommand\eq[2]{{\ifdraft{\ \tt [#1]}\else\ignorespaces\fi}\begin{equation}\label{#1}{#2}\end{equation}}
\newcommand {\equ}[1]{\eqref{#1}}

\font\sb = cmbx8 scaled \magstep0
\font\sn = cmssi8 scaled \magstep0

\long\def\combarak#1{\ifdraft{\color{red}\sn #1 }\else\ignorespaces\fi}
\long\def\comdima#1{\ifdraft{\color{green}\sb #1 }\else\ignorespaces\fi}

\newif\ifdraft\drafttrue

\draftfalse

\newcommand\hs{homogeneous space}
\newcommand\da{Diophantine approximation}

\begin{document}

\title{Pointwise equidistribution with an error rate and with respect to unbounded functions}

\author{Dmitry Kleinbock}
\address{Department of Mathematics, Brandeis University, Waltham MA 02454, USA}
\email{kleinboc@brandeis.edu}

\author{Ronggang Shi}
\address{School of Mathematical Sciences, Tel Aviv University, Tel Aviv 69978, Israel, and
School of Mathematical Sciences, Xiamen University, Xiamen 361005, PR China}
 \email{ronggang@xmu.edu.cn}

\author{Barak Weiss}
\address{School of Mathematical Sciences, Tel Aviv University, Tel Aviv 69978, Israel}
\email{barakw@post.tau.ac.il}


\date{}



\begin{abstract} 
Consider $G=\SL_{ d }(\mathbb R)$ and $ \Gamma=\SL_{ d }(\mathbb Z)$.
It was recently shown by the second-named author \cite{s} that for
some diagonal subgroups $\{g_t\}\subset G$ and unipotent subgroups
$U\subset G$, $g_t$-trajectories of almost all points on all
$U$-orbits on $G/\Gamma$ are equidistributed  with respect to
continuous compactly supported functions $\varphi$ on $G/\Gamma$. In
this paper we strengthen this result in two directions: by exhibiting
an error rate of equidistribution when $\varphi$ is smooth and
compactly supported, and by proving equidistribution with respect to
certain unbounded functions, namely Siegel transforms of Riemann
integrable functions on $\R^d$.  
For the first part we use a method based on effective double
equidistribution of $g_t$-translates of $U$-orbits, which generalizes
the main result of \cite{km12}. The second part is based on Schmidt's
results on counting of lattice points. Number-theoretic
consequences  involving spiraling of lattice approximations, extending
recent work of Athreya, Ghosh and Tseng \cite{agt1}, are derived using the equidistribution
result. 
\end{abstract}

\maketitle

\section{Introduction}\label{sec;intro}

\markright{} Fix an integer $d \ge 2$,  let
$G=\SL_{ d }(\mathbb R)$ and $ \Gamma=\SL_{ d }(\mathbb Z)$,  and
denote by $ \Xn$ the \hs\  $G/\Gamma$, which can be identified with
the 
space of unimodular lattices in $\R^d$. 
The group $G$ acts on $\Xn$ by
left translations preserving the probability measure $\mu$ induced by Haar measure. This action
has been intensively studied due to its intrinsic interest as a
dynamical system and for its number-theoretic applications. See
\cite[Chapter 5]{survey} for a survey of this topic. 

Let $D=\{g_t\}$ be a 
one-parameter subgroup of $G$, and let $\varphi$
be a real-valued function on $\Xn$. We say that $\Lambda
\in \Xn$ is {\em $(D^+, \varphi)$-generic} if 
\eq{generic}{
\lim_{T\to \infty}\frac{1}{T}\int_0^T \varphi(g_t\Lambda)  \dd t= \int_{\Xn}
\varphi \dd\mu\,, 
}
and, for a collection $\mathcal{S}$ of functions, that $\Lambda$ is
{\em $(D^+, \mathcal{S})$-generic}  if it is $(D^+, \varphi)$-generic
for every $\varphi \in \mathcal{S}$. Let $C_c(\Xn)$ denote the space of
continuous compactly supported functions on $\Xn$. It is a well-known consequence of
Moore's ergodicity theorem and Birkhoff's ergodic theorem  that
for any unbounded subgroup $D$ of $G$, $\mu$-almost every $\Lambda\in\Xn$ is $\big(D^+,
C_c(\Xn)\big)$-generic. Also, using effective mixing estimates for the
$D$-action on $X$ one can conclude that whenever $\varphi$ belongs to
the space $C_c^\infty(\Xn)$ of smooth compactly supported functions on
$\Xn$, the convergence in \equ{generic} takes place with a certain
rate for $\mu$-a.e.\ $\Lambda\in\Xn$. More precisely, see
\cite[Theorem 16]{Kach} or \cite[Theorem 4(vii)]{Gap}, for any
$\varepsilon > 0$ and  $\mu$-a.e.\ $\Lambda\in\Xn$ one has  
\eq{eq;main rate in X}{
\frac{1}{T} \int_0^T \varphi \big (g_t  \Lambda \big)
\dd t=\int_{\Xn} \varphi \dd\mu+
o(T^{-1/2}\log^{\frac32+\varepsilon} T).}
(The notation $f(T) =
  o\big(g(T)\big)$ means $\lim_{T\to\infty} \frac{|f(T)|}{g(T)} = 0$.)

\smallskip
Now let $U^+$ be  the {\em unstable horospherical subgroup of $G$ relative to\/} $D^+$, defined by
\eq{ehs}{U^+ := \{g \in G : g_{-t}
gg_t \to e \text{ as }t \to\infty\}.
}
Assume in addition that  
$D$ is diagonalizable.
Then, using a local decomposition of $G$ as the product of its
unstable, neutral and stable horospherical subgroups with respect to
$D^+$, see e.g.\ \cite[\S 1.3]{km96},  it is easy to conclude that for
any $\Lambda \in \Xn$  the element  $g\Lambda$ is $\big(D^+, 
C_c(\Xn)\big)$-generic   for Haar-a.e.\ $g\in U^+ $.

Recently in \cite{s}, as a corollary of a more
general result, the second-named author obtained a similar conclusion
for some proper subgroups of $U^+$ -- namely, for the class of
so-called {\em $D^+$-expanding\/} subgroups, introduced in \cite{KW2}
(see also \cite{SW} for some related ideas).  More precisely, the
following is a special case of  \cite[Theorem 1.2]{s}:  
if $D$ is a diagonalizable one-parameter
subgroup of $G$ and $U$  is a connected $D^+$-expanding abelian subgroup of $U^+$, then
 \eq{shi}{  
 \forall\,\Lambda \in \Xn, \ 
  g\Lambda\text{  is $\big(D^+,
C_c(\Xn)\big)$-generic   for Haar-a.e.\ }g\in U.}

In the present paper we consider a specific family of pairs $(D,U)$,
where $D\subset G$ is  one-parameter    and $U\subset G$ is
$D^+$-expanding,  which are  important for number-theoretic
applications. Namely, take $m,n\in \N$ with $m+n = d$,  denote by  $M$
the space 
 of ${\m}\times {\n}$ real matrices, and  consider 
\eq{defU}{U := u(M)\ \text{ where }
u: M \to G \text{ is given by }
u(\vartheta):=\left(\begin{array}{cc}
1_{\m} & \vartheta \\
0 & 1_{\n}
\end{array}\right)
}
(here and hereafter $1_k$ stands for the identity matrix of order $k$).  Also fix two `weight vectors' 
\begin{align*} \mathbf a=(a_1,\ldots, a_{\m})\in &\mathbb R^{\m}_{>0} \quad \text{and}\quad
\mathbf b=(b_1, \ldots, b_{\n})\in \mathbb R^{\n}_{>0} \\
\qquad \mbox{such that }&\sum_{i=1}^{\m} a_i =\sum_{i=1}^{\n} b_i=1.
\end{align*}
We will refer to the case 
$${\ba = \bm =   (1/{\m},\dots,1/{\m}) , \  \bb = \bn =  (1/{\n},\dots,1/{\n})}$$ as
the case of {\em equal weights.} Then take 
\eq{defD}{D = \{g_t\},\  \text{ where }g_t=\mathrm{diag}(e^{a_1t},
  \ldots, e^{a_{\m}t}, e^{-b_1t},\ldots, e^{-b_{\n}t} ).  
}

It is easy to see that  
for any choice of weight vectors $\ba,\bb$,  the group $U$ as in
\equ{defU}  is contained in the unstable horospherical subgroup
relative to $D^+$  (and coincides with it in the case of equal
weights). It was shown in \cite{KW2} that $U$ is  $D^+$-expanding for
any $D$ as in \equ{defD}, which,  in view of \cite{s}, implies
\equ{shi}. 
For the rest of the paper we will fix arbitrary weight vectors
$\ba,\bb$ and choose $D$ as in \equ{defD}. 

\smallskip

Our first main result gives an analogue of \equ{eq;main rate in X}  for almost all points on $U$-orbits: 

\begin{thm}\label{thm;main rate}
Let  $\Lambda \in \Xn$, $\varphi\in C_c^\infty(\Xn)$ and $\varepsilon>0$ be given.
Then for  almost every $\vartheta\in M$ 
\begin{align}\label{eq;main rate goal}
\frac{1}{T} \int_0^T \varphi \big (g_t u(\vartheta)\Lambda \big)
\dd t=\int_{\Xn} \varphi \dd\mu+
o(T^{-1/2}\log^{\frac32+\varepsilon} T).
\end{align}
\end{thm}

Here and hereafter `almost every $\vartheta$' means almost every with
respect to Lebesgue measure on $M \cong \R^{{\m}{\n}}$; this
corresponds to taking a Haar measure on $U$.

\smallskip
Our proof of Theorem \ref{thm;main rate} is completely
independent of \cite{s}, and, in view of the density of
$C_c^\infty(\Xn)$   in $C_c(\Xn)$, provides an alternative
demonstration \equ{shi} for the case \equ{defU}--\equ{defD}.
It is based on the following `effective double equidistribution' of
$g_t$-translates of $U$, generalizing the 
`effective equidistribution' result of \cite{km12}:

\begin{thm}\label{thm;mixing}
There exists $\delta >0$ such that the following holds. 
Given $f\in C_c^\infty (M)$ and $\varphi,\psi \in C_c^\infty(\Xn) $
and 
 a compact subset  $L$  of $\Xn$, there exists $C
>0$ such that 
for any $\Delta, \Lambda \in L$ and $t, {w}\ge 0$  one has
\begin{align}\label{eq;mixing rate}
\left|
I_{\Delta,\Lambda{,f,\varphi,\psi}}{(t, w)} -\int_M f \dd\vartheta  \int_{\Xn} \varphi\dd\mu \int_{\Xn}\psi\dd\mu
\right|
\le C e^{-\delta \min (t,  {w}, |{w}-t|)},
\end{align}
where
\begin{align}\label{eq;mixing goal}
I_{\Delta,\Lambda{,f,\varphi,\psi}}(t,{w}) :=
\int_{M} f(\vartheta)\varphi
\big(g_tu(\vartheta)\Delta\big)\psi \big(g_{w}u(\vartheta)\Lambda\big)\dd  \vartheta.
\end{align}
\end{thm}

Under the same assumptions one can also prove `effective $k$-fold equidistribution' of $g_t$-translates of $U$ for all
$k\in\N$. We hope to return to this topic elsewhere. 
See
also
\cite{BEG} and   \cite{konst} for related results on `effective $k$-mixing'.
To derive pointwise equidistribution of $g_t$-trajectories of points
on $U$-orbits with an error rate from  
effective double equidistribution of $g_t$-translates of $U$ we use a
method borrowed from Schmidt's work \cite{s60}  
and originally due to Cassels \cite{cassels}. It is also similar
  to the argument used in \cite{Gap} to derive the rate of convergence
  in Birkhoff's Theorem from the rate of decay of correlations. As a
  byproduct of the method, we show in \S\ref{sec: effective} (Remark
  \ref{alternative proof}) how the estimate  \equ{eq;main rate in X}
  for $\mu$-a.e.\ $\Lambda$ can be derived from the exponential mixing
  of the $D$-action on $\Xn$.

\smallskip

The second main theme of the present paper is establishing $(D^+,
\varphi)$-genericity with respect to some unbounded functions
$\varphi$. It follows from Birkhoff's Theorem that for any 
 $\varphi \in L^1(X,\mu)$,  $\mu$-almost every $\Lambda
\in \Xn$ is $(D^+, \varphi)$-generic; however a passage from $\mu$-a.e.\ $\Lambda
\in \Xn$ to Haar-almost all points on $U^+$-orbits requires some regularity of $\varphi$, such as 
continuity and Lipschitz property away from a compact subset. 
When $U^+$ is replaced by its proper subgroup, the situation becomes more complicated. In particular,  our method of proof of Theorem \ref{thm;main rate}, as well as the argument in \cite{s}, 
do not work for unbounded functions.
Yet we are able to prove a partial result for the setup
\equ{defU}--\equ{defD} and explore its number-theoretic consequences.  

In order to control the rate of growth at infinity of unbounded
functions on $\Xn$, following \cite{emm98},  
let us introduce a function measuring `penetration into 
the thin part of $X$'. For a lattice $\Lambda \in X$, given a subgroup
$\Lambda' \subset \Lambda$, let 
$d(\Lambda')$ denote the covolume of $\Lambda'$ in
$\mathrm{span}_{\R}(\Lambda')$ (measured with respect to the standard
Euclidean structure on $\R^d$). We denote 
$$
\alpha(\Lambda) = \max \left\{d(\Lambda')^{-1
} : \Lambda'
\text{ a subgroup of } 
\Lambda \right\}.
$$
 It is well-known that this maximum is attained, and defines a proper
map $X \to [1, \infty)$. 

Now let us denote by $C_{\alpha}(X)$  the space of  functions
$\varphi$ on $X$ satisfying the following two properties: 
\begin{itemize}
\item[($C_\alpha$-$1$)] $\varphi$ is continuous except on a set of $\mu$-measure zero;
\item[($C_\alpha$-$2$)] the growth of $\varphi$  
is majorized by  $\alpha$, namely  there exists $C>0$ 
such that for all $\Lambda \in X $, we have 
$$
|\varphi(\Lambda)| \leq C\alpha(\Lambda)
$$
(in particular, $\varphi$ is bounded on compact sets).
\end{itemize}

One can show that the space $C_{\alpha}(X)$ contains certain unbounded
functions which often arise in number-theoretic applications.
Recall
that $f: \mathbb R^{ d }\to \mathbb R$ is said to be 
\emph{Riemann 
integrable} if $f$ is bounded with bounded support
and continuous except on  a set of Lebesgue measure zero.  
  {For such $f$, we} define a function $\widehat f$ on
$\Xn$ by 
\eq{eq: Siegel transform}{
\widehat f\left ( \Lambda \right):=\sum_{\mathbf v\in \Lambda \nz} f(\mathbf v).
}
The {\em Siegel integral formula}, established in \cite{siegel}, asserts that for any $f$ as above,
$\int_{\Xn} \widehat f \dd\mu=\int_{\mathbb R^{ d }} f$.
We show in Lemma \ref{lem: Blichfeldt} that for any Riemann integrable
$f$, the function $\widehat{f}$ satisfies conditions ($C_\alpha$-$1$) and 
($C_\alpha$-$2$). 
In fact, up to constants, when $f$ is non-negative and 
  {nonzero on a set of positive measure}, the order of growth of $\widehat f$
is precisely the same as that of $\alpha$. 
As a consequence, in view of the Siegel integral formula, $C_\alpha(X)$ is
contained in $L^1(X,\mu)$. 

\smallskip

For our second main result,
we denote by $\Lambda_0$ the  
standard lattice $\Z^d\subset \R^d$.

\begin{thm}\label{thm;main} Let $U$ and $D$ be as in \equ{defU}--\equ{defD}.
Then for almost every $\vartheta\in M$, 
$u(\vartheta)\Lambda_0$ is 
$\big(D^+, C_{\alpha}(X)
\big)$-generic. 
\end{thm}

Since $\big(D^+,C_c(\Xn)\big)$-genericity has already been proved,
the
main additional point in
the proof of Theorem \ref{thm;main} is  to obtain 
 upper bound{s} for Birkhoff averages of {restrictions of non-negative $\varphi\in 
 C_\alpha(X)$ to complements of large compact subsets of $X$.}
To this end we
employ a lattice point counting result of Schmidt \cite{s60}. 
Lattice 
point counting  was
used for a similar purpose in \cite{ms14},
and the
connection between Schmidt's result and the action of $D^+$ was already noted in
\cite{apt} in the equal weights case.
Note that in our result we assume that the lattices are of the form
$ u(\vartheta)\Lambda_0$; 
we expect a similar result to be true if $\Lambda_0$ is replaced by
any other lattice in $\Xn$. However our proof does not yield this more
general statement. 

\smallskip

Let us now explain the number-theoretic implications of Theorem \ref{thm;main}. 
In Diophantine approximation  one interprets $\vartheta \in M$ as a
system of ${\m}$ linear forms in $n$ variables and studies how close to
integers are the values $\vartheta\bq$ of those forms at integer vectors $\bq$. 
During recent years there have been many developments in {\em \da\
  with weights\/}, which allows to treat individual  components of
$\bq$ and $\vartheta\bq$ differently.     
This is done by choosing weight vectors $\ba,\bb$ and considering
`weighted quasi-norms' introduced in  \cite{dima weights}: 
\[
\|\mathbf x\|_{\mathbf a}=\max_{1\le i\le {\m}} |x_i|^{\frac{1}{a_i}} 
\quad \left(\text{resp.}~\|\mathbf y\|_{\mathbf b}=\max_{j\le i\le {\n}}
|y_j|^{\frac{1}{b_j}} \right).
\]
It is a consequence of the general version of the Khintchine-Groshev Theorem proved by Schmidt 
\cite{s60}
 that for  a.e.\  $\vartheta \in M$ 
and any $c > 0$ there are infinitely many
solutions $(\bp,\bq)\in\Z^{\m}\times\Z^{\n}
$ to the
inequality 
\eq{kg1}{\|\vartheta\bq - \bp\|_{\ba} < \frac c{\|\bq\|_{\bb}}.}
More precisely, the number of integer
solutions of \equ{kg1} with $\bq$ satisfying
\eq{kg2}{1\le \|\bq\|_{\bb} < e^T}
has the same asymptotic growth
as $ 2^d c T$  (here and hereafter we say that $f(T)$ and $g(T)$ {\em have the same
 asymptotic growth\/}, denoted $f(T) \thicksim g(T)$, if $\displaystyle{\lim_{T \to \infty} 
f(T)
/g(T)
= 1}$). 
We remark that the set of nonzero integer solutions of \equ{kg1}--\equ{kg2}
 is in one-to-one correspondence with the intersection of the lattice
$${u(\vartheta) \Lambda_0 = 
\left\{ \left( \begin{matrix}\vartheta \mathbf{ q} - \mathbf{p} \\
      \mathbf{q} \end{matrix} \right) : \mathbf{p} \in \Z^{\m}, \,
  \mathbf{q} \in \Z^{\n} \right\}
}$$
with the set 
$$
E_{ T, c} :=
\left\{(\bx, \by) \in \R^{\m} \times \R^{\n}:\|\bx\|_\ba  < \frac c{\|\by\|_\bb},\ \  
1\le \|\by\|_\bb<e^{T} \right\},
$$
and an elementary computation shows that the volume 
of this set is equal to $ 2^d c T$.

\smallskip
A finer question concerning  directions of  vectors
$\vartheta\bq - \bp$ and $\bq$ was addressed in a recent paper \cite{agt1} by Athreya,  Ghosh and Tseng. In the equal weights case 
 they considered integer solutions of \equ{kg1}--\equ{kg2} in addition satisfying 
 \eq{unweightedspherical}{\pi( \vartheta\bq - \bp) \in  A\text{ and }\pi( \bq) \in B ,}
 where $\pi$ stands for the radial
projection from $\R^{\m}$ and $\R^{\n}$ to the unit spheres $\mathbb
S^{{\m}-1}$ and  $\mathbb S^{{\n}-1}$ respectively, and  $A \subset \mathbb 
S^{{\m}-1}, \ B \subset \mathbb S^{{\n}-1}$ are two measurable
subsets. When  the boundaries of $A$ and $B$ are of measure zero, they
showed that  
 for almost every  $\vartheta \in M$ the 
number  of  solutions of \equ{kg1}--\equ{unweightedspherical} (with $\ba = \bm$ and $\bb = \bn$)
has the same asymptotic growth as the volume of the set 
\begin{equation*}
\left\{(\bx, \by) \in \R^{\m} \times \R^{\n} :  \begin{aligned}\|\bx\|^{\m}  <  \frac c{\|\by\|^{{\n}}}&,\ \  
1\le \|\by\|^{\n}<e^T\\ \pi( \bx) \in  A&, \ \  
\pi( \by) \in B\end{aligned} 
\right\}
\end{equation*}
(here $\|\cdot\|$ stands for the supremum norm; 
note that \cite{agt1} deals only with the case $m=1$ and uses
  Euclidean norm instead of the supremum norm, but the argument can be
  applied to an arbitrary $m$ and arbitrary norm). 

In the case of arbitrary weight vectors it no longer makes sense to project radially, since these
projections accumulate near directions corresponding to
smallest (resp., largest) weights. We remedy this as follows. Let
$F_{\mathbf a t}$ be the $\mathbf a$-weighted  flow on $\mathbb R^{\m}$
defined by  
 \[
F_{\mathbf a t}( \mathbf x):=\left(e^{a_1t}x_1, \ldots, e^{a_{\m}t}x_{\m} \right)
 \]
(note that in the case of equal weights, these are just homotheties of
$\R^{\m}$). Similarly we define the transformation $F_{\bb t}$ of $\R^{\n}$. Then 
for  nonzero $\mathbf x\in \mathbb R^{\m}$ and $\by \in \mathbb R^{\n}$  let 
\eq{defpi}
{\pi_{\ba}( \mathbf x):=\{ F_{\mathbf a t}(\mathbf x): t\in \mathbb R\}
  \cap \mathbb S^{{\m}-1}, \ \pi_{\bb}( \mathbf y):=\{ F_{\mathbf b t}(\mathbf y): t\in \mathbb R\}
  \cap \mathbb S^{{\n}-1}}
 (these intersection point are clearly unique),
replace
\equ{unweightedspherical} with \eq{weightedspherical}{\pi_\ba(
  \vartheta\bq - \bp) \in  A\text{ and }\pi_\bb( \bq) \in B,} 
  and define 
$$
E_{ T, c}(A,B) := 
\left\{(\bx, \by) \in \R^{\m} \times \R^{\n}: \begin{aligned}\|\bx\|_\ba  < \frac c{\|\by\|_\bb}&,\ \  
1\le \|\by\|_\bb<e^{T}\\ \pi_\ba( \bx) \in  A&,\ \  
\pi_\bb( \by) \in B\end{aligned}\right\}.
$$
Using  Theorem \ref{thm;main}, we obtain: 

\begin{thm}\label{thm;diophantine}
Let  $c > 0$ and measurable  subsets $A \subset \mathbb
S^{{\m}-1}, \ B \subset \mathbb S^{{\n}-1}$ with boundaries of measure zero be given.
Then for a.e.\  $\vartheta \in M$, as $T\to\infty$,  the number of
integer solutions  to \equ{kg1}, \equ{kg2} and \equ{weightedspherical}
has
the same asymptotic growth as
the volume of $E
_{ T, c}(A,B)$. 
\end{thm}

Note that the above counting result does not follow from the
techniques of \cite{s60, S60a}. See also \cite{agt} for some related results. The
reduction of Theorem \ref{thm;diophantine} to Theorem \ref{thm;main} is based on
the observation that the number of
integer solutions  to \equ{kg1}, \equ{kg2} and \equ{weightedspherical} is equal
 to $\sharp \big( u(\vartheta)\Lambda_0\cap E_{T, c}(A, B)\big)$, that
 is, to $\widehat f$,  where $f$ is a certain Riemann integrable
 function on $\R^d$.

 \smallskip
The structure of the paper is as follows: in \S\ref{sec:double} we
prove Theorem \ref{thm;mixing}, and  use it  in  \S\ref{sec:
  effective} to  prove Theorem \ref{thm;main rate}. \S\ref{sec;count}
contains a discussion of Schmidt's asymptotic formula (Theorem
\ref{thm;schmidt use}) and some auxiliary results, which are needed in
the final section of the paper, where Theorems \ref{thm;main} and
\ref{thm;diophantine} are proved.  
 \smallskip
 
\noindent{\bf Acknowledgements.}
The support of grants 
NSFC (11201388), NSFC (11271278),
 ERC starter grant DLGAPS 279893, NSF DMS-1101320 and BSF 2010428 is gratefully
 acknowledged. 
 We would like to thank Manfred Einsiedler for a discussion of
 effective double equidistribution and MSRI for its hospitality during
 Spring 2015.

\section{Effective double equidistribution}\label{sec:double}
In this section we prove Theorem \ref{thm;mixing}. We first recall several facts from the paper
\cite{km12}. Its main result, i.e.\ effective equidistribution of
$g_t$-translates of $U$, can be stated as follows:

\begin{prop}[\cite{km12}, Theorem 1.3]\label{prop;one mixing}
There exists $\delta_1>0$ such that for any $f\in C_c^\infty(M)$,
$\varphi \in C_c^\infty(\Xn)$, and 
for any compact $L\subset \Xn$ there exists $C_1=C_1(f, \varphi, L)$
such that for any $\Lambda\in L$ and  
$t\ge 0$,
\begin{align}
\left|
\int_M f(\vartheta)\varphi\big(g_t u(\vartheta)\Lambda\big)\dd \vartheta
 -\int_M f \dd \vartheta \int_{\Xn} \varphi \dd \mu
\right|
\le C_1 e^{-\delta_1 t}.
\end{align} 
\end{prop}
Let  
$$b = \min \left\{\frac{a_i}{{\n}}, \frac{b_j}{{\m}} : 1\le i\le {\m},\ 1\le j\le
  {\n}\right \},
$$ 
and let 
\begin{align}
g_t'=\left (
\begin{array}{cc}
e^{{\n}bt}1_{\m} & 0 \\
0 & e^{-{\m}bt} 1_{\n} 
\end{array}
\right),
\quad
g_t''=g_tg'_{-t}.
\end{align}
Note that $ g_t'$ (after reparametrization) corresponds to the case of
equal weights; in particular, $U$ is the unstable horospherical
subgroup relative to $\left\{g_t': t\ge 0 \right\}$.  On the other hand $ g_t''$  
is of the form
$$\mathrm{diag}\left(e^{a'_1t}, \ldots, e^{a'_{\m}t},
  e^{-b'_1t},\ldots, e^{-b'_{\n}t} \right )
$$
where $a_i'$ and $b_j'$ are nonegative.

Let $B_r^G$ be the ball of radius $r$ centered at the identity element
of $G$ with respect to the metric induced by some right invariant 
Riemannian metric. For $\Lambda\in \Xn$ we let 
$\pi_\Lambda: G\to \Xn$ be the map $g\mapsto g\Lambda$, and define
\[
\mathrm{Inj}_\varepsilon:=\left\{\Lambda\in \Xn:
  \pi_\Lambda|_{B_\varepsilon^G} \mbox{ is injective}\right \}. 
\]
Also let $B_r$ denote  the ball of radius $r$ centered at zero with respect to 
the Euclidean norm on $M\cong \mathbb R^{mn}$. 
The following quantitative non-divergence result is a combination of \cite[Corollary 3.4]{km12}  
and \cite[Proposition 3.5]{km12}:
\begin{prop}\label{prop;quantitative}
Let $L\subset \Xn$ be compact and let $r>0$. There exists 
$t_0=t_0(L, r)>0$ and 
$C_3>0$ such that for every 
$0<\varepsilon<1$, $\Lambda\in L$, $s\ge 0$ and $t\geq t_0$ one has
\[
\left| \left\{\vartheta\in B_r: g_s''g_t u(\vartheta)\Lambda\not \in
    \mathrm{Inj}_\varepsilon \right\} \right|\le 
C_3\varepsilon ^ {{1}/{d^4}} \left|B_r \right|.
\]
\end{prop}
Here and hereafter for a measurable subset $S$ of a Euclidean space we use 
$|S|$ to denote the Lebesgue measure of $S$.  
For $h\in C_c^\infty(M)$ and a nonnegative 
integer $k$ we define the $k$-th Sobolev norm of $h$ to be 
\[
\|h\|_k=\max_{|\beta|\le k} \left\|\partial ^\beta h \right\|_{L^2(M)}
\]
where $|\beta|$ is the order of the multi-index $\beta$.

We will also need the following effective equidistribution estimate for the  $g_t'$-action:
 
\begin{prop}[\cite{km12}, Theorem 2.3]\label{prop;injective}
There exist $\delta_2, r_0 > 0$ and $k\in \mathbb N$ with the
following property:  for any $\psi\in C_c^\infty(\Xn)$ 
there exists $C_2 > 0$
such that for any $0<r<r_0$, any $ h\in C_c^\infty (M)$ with
$\mathrm{supp}(h)\subset B_r$, any $\Lambda\in 
\mathrm{Inj}_{2r}$  and any $t\ge 0$ one has
 \[
 \left|
 \int_M h(\vartheta)\psi\big(g_t'u(\vartheta)\Lambda
\big )\dd\vartheta
-\int_M h\int_{\Xn}\psi 
 \right|\le
  C_2\left(r^{-k}\|h\|_ke ^{-\delta_2 t}+
  r\int_M |h|
  \right).
 \]
\end{prop}

Note that the statement is not precisely the one which appears in
\cite{km12} but can be deduced from it by taking $k$ large enough.

\begin{proof}[Proof of Theorem \ref{thm;mixing}] Recall that we are
  given a compact subset  $L$  of $\Xn$, $ \varphi,\psi \in
  C_c^\infty(\Xn) $ and  $f\in C_c^\infty (M)$. 
In this proof, the notation 
$y\ll x$  will mean that $y \leq Cx$ where $C$ is a
constant independent of $x$, but which could depend on $f, \varphi,
\psi$ or $L$.

 Let $\delta_0=\min(\delta_1, \delta_2)$ where $\delta_1$
and $\delta_2$ are given 
in Propositions \ref{prop;one mixing} and \ref{prop;injective} respectively. 
We also fix $k\in \mathbb N$ so that Proposition \ref{prop;injective} holds. 
Without loss of generality we assume that $k\ge mn$. Now let 
 $$a:=\min\{a_i+b_j: 1\le i\le  {\m}, 1\le j\le {\n} \}$$
and 
\eq{eq: choice of delta}{
\delta:=\min \left( \frac{\delta_0}{2(1+3k)}, \frac{a}{2}\right).
}
Let 
 $\Delta$,   $\Lambda$,   $ t$ and   ${w}$ be as in the statement of
Theorem \ref{thm;mixing}.  
Put 
$$r:=e^{-\delta |{w}-t|} \ \text{ and }  s:=\frac{t + {w}}{2}.$$
We can assume with no loss of generality
that ${w} \geq t \geq t_0$, where $t_0$ is as in Proposition
\ref{prop;quantitative}, and moreover that $r< r_0$, where $r_0$ is as in
 Proposition \ref{prop;injective}. 
According to \cite[Lemma 2.2]{km12}
there exists $h\in C_c^\infty (M)$ such that supp$(h)\subset B_r$, $h\ge 0$,
$\int_M h (\vartheta)\dd \vartheta=1$ and $\|h\|_k\ll r^{-2k}$.  
Given $\zeta\in M$, we define  $\zeta_1, \zeta_2\in M$ by the formulae
\[
u(\zeta_1): =g_{t-s}''g_{-s}u(\zeta) g_{s}g_{s-t}'', \quad
u(\zeta_2): =g_{t-s}''g_{t-s}u(\zeta) g_{s-t}g_{s-t}'',
\]
which imply 
\eq{eq: with}{
g_t u(\zeta_1) = u(\zeta_2) g_t \ \text{ and } \  g_{{w}} u(\zeta_1) =
g_{s-t}' u(\zeta) g_s g''_{s-t}. 
}

Let $I_{\Delta,\Lambda{,f,\varphi,\psi}}(t, {w})$ be as in \equ{eq;mixing goal}. Then, 
{using Fubini's theorem 
and, for each $\zeta$, making a change of variables 
 $\vartheta\mapsto \vartheta+\zeta_1$, 
we find in view of \equ{eq: with} 
that 
\begin{equation}
\begin{aligned}
&I_{\Delta,\Lambda{,f,\varphi,\psi}}(t, {w})=  \int_M  f(\vartheta) \varphi\big(g_tu(\vartheta)\Delta\big)
\psi\big(g_{w} u(\vartheta)\Lambda\big)  \dd \vartheta \int_M h (\zeta)\dd \zeta\\ =&
{\int_M \int_M} \  f(\vartheta+\zeta_1)
\varphi\big(u(\zeta_2)g_t u(\vartheta)\Delta\big)\psi\big(g_{s-t}'u(\zeta)g''_{s-t}g_s u(\vartheta)\Lambda\big)
h(\zeta)\dd \vartheta\dd \zeta.
\label{eq;q 2}
\end{aligned}
\end{equation}
}

Note  that for all $\zeta\in \mathrm{supp}(h)$ we have 
  $$\|\zeta_1\| \ll  e^{-sa} \|\zeta\| \ll
  e^{-\frac{{w}-t}{2}a} \stackrel{\equ{eq: choice of delta}}{\ll}
  e^{-\delta({w} -t)} ,$$ 
and similarly  $\|\zeta_2\| \ll e^{-\delta({w}-t)}$. 
Denote
$$
{\Psi}(\vartheta) := \int_M \psi\left(g_{s-t}'u(\zeta)g_{s-t}''g_s u(\vartheta)\Lambda\right)
h(\zeta) \dd \zeta.
$$
Then, by approximating the function $f(\vartheta+\zeta_1)
\varphi\big(u(\zeta_2)g_t u(\vartheta)\Delta\big)$ by $f(\vartheta) 
\varphi\big(g_t u(\vartheta)\Delta\big)$ in (\ref{eq;q 2}) we 
get
\begin{align}
\left|
I_{\Delta,\Lambda{,f,\varphi,\psi}}(t, {w})
- 
\int_M 
{\Psi}(\vartheta)
 f(\vartheta) \varphi\big(g_t u(\vartheta)\Delta\big)\dd \vartheta\right| \ll
e^{-\delta({w}-t)} 
.
\label{eq;effective 1}
\end{align}

Let $r_1>0$ be such that $\mathrm{supp} \, f \subset B_{r_1}$, let 
$\varepsilon=e^{-\delta ({w}- t)d^4}$, and {denote}
 $$E:=\{\vartheta\in B_{r_1}: g_{s-t}''g_s
 u(\vartheta)\Lambda\in \mathrm{Inj}_\varepsilon\}.$$  
It follows from 
Proposition \ref{prop;quantitative} that  
 \begin{align*} 
|B_{r_1}\ssm  E|\ll e^{-\delta({w}-t)}|B_{r_1}|,
\end{align*}
and hence 
\begin{align}\label{eq;effective 11}
\left |
\int_{B_{r_1}\ssm E}{\Psi}(\vartheta)f(\vartheta)\varphi\big(g_t u(\vartheta )\Delta\big)\dd \vartheta
\right|
\ll e^{-{\delta}({w}-t)}.
\end{align}
By Proposition \ref{prop;injective} 
for $\vartheta\in E$ (with $g_{s-t}'' g_s u(\vartheta)
\Lambda$ in place of $\Lambda$) one has
$${
\left| {\Psi}(\vartheta) - \int_{\Xn} \psi 
\right |
\ll r + r^{-k}
  \|h\|_k e^{-\delta_2(s-t)}
\ll e^{-\delta({w}- t)}+e^{-(\delta_0/2-3k\delta){({w}-t) } }
,
}
$$
{which, in view of 
\equ{eq: choice of delta} and  (\ref{eq;effective 11}), implies} 
\begin{equation}
\label{eq;effective 2}
{
 \begin{aligned}
\left | \int_{M}{\Psi}(\vartheta)f(\vartheta)\varphi\big(g_t u(\vartheta )\Delta\big)\dd \vartheta  
 -  \int_{M}f(\vartheta)\varphi\big(g_t u(\vartheta )\Delta\big)\dd \vartheta\int_{\Xn}\psi
\right|
\\
 \ll
e^{-{\delta({w}-t)}}+e^{-\frac{\delta_0({w}-t) } {2-3k\delta} }
\ll e^{-\delta({w}-t)}
. \qquad\qquad\qquad
\end{aligned}
}
\end{equation}
On the other hand {from Proposition \ref{prop;one mixing} 
one gets}
\begin{align}\label{eq;effective 3}
\left|
\int_{M}f(\vartheta)\varphi\big(g_t u(\vartheta )\Delta\big)\dd \vartheta
-\int_M
f\int_{\Xn}\varphi
\right|\ll
e^{-\delta_0 t} 
.
\end{align}
Combining (\ref{eq;effective 1}), (\ref{eq;effective 2})   and (\ref{eq;effective 3}), one arrives at (\ref{eq;mixing rate}).
\end{proof}

\section{Pointwise equidistribution with an error rate}
\label{sec: effective}
In this section we prove Theorem \ref{thm;main rate}.  
The method works in a general framework as follows:
\begin{thm}\label{thm;effective}
Let $(Y, \nu)$  be a probability space, and let $F:Y\times \R_+\to \R$ be a bounded measurable function.
Suppose  there exist $\delta >0$ and $C>0$ such that for any $w
\ge t \ge 0$,
\begin{align}\label{eq;effective rate}
\left |
\int_Y{F(x,t)}{F(x,w)}\dd \nu( x)
\right|
\le Ce^{-\delta \min(t, {w}-t)}.
\end{align}
Then given $\varepsilon >0$ we have 
\begin{align}\label{eq;effective goal}
\frac{1}{T}\int_0^T {F(y,t)}\dd t=o
(T^{-1/2}\log^{{\frac32+\varepsilon} } T)
\end{align}
for $\nu$-almost every ${y}\in Y$. 
\end{thm}

We begin with some lemmas. In the statements below the notation and
assumptions are as in Theorem \ref{thm;effective}. 

\begin{lem}\label{lem;effective help}
 Let $[b,c]$ be a closed 
interval in $[0,\infty)$. Then 
\[
\int_Y \left(
\int_b^c{F(x,t)}\dd t
\right)^2 \dd \nu(x)
\le { 4C\delta ^{-1}(c-b)}.
\]
\end{lem}

\begin{proof} The left hand side can be written as
\begin{align}
  & \int_Y \left(
\int_b^c{F(x,t)}\dd t\right)^2 \dd \nu(x)  \notag\\
 = &\int_b^c \int _b^c  \left
 [\int_Y{F(x,t)}{F(x,w)} \dd\nu (x)\right]\dd {w} 
\dd t  &&\mbox{by Fubini}\notag\\
 \le & 2\int_b^c \int_t^c  \left |\int_Y{F(x,t)}
{F(x,w)} \dd\nu (x)\right|\dd {w} \dd t  \notag\\
 \le& 2C\int_b^c  \left [\int_t^{c}
(  e^{-\delta ({w}-t)}+
   e^{-\delta t})\dd {w}  
 \right]\dd t  &&\mbox{by (\ref{eq;effective rate})} \notag  \\
  \le  & 2C\int_b^c \big({\delta^{-1} }+e^{-\delta
    t}(c-t )\big)\dd t \notag\\
\leq  & 2C\int_b^c \big({\delta^{-1} }+e^{-\delta
    t}(c-b )\big)\dd t
\    \le\ 
    {4C\delta^{-1}(c-b)}, \notag
\end{align}
and the proof is finished.
\end{proof}

For a positive integer $s$ we let $L_s$ be the set of intervals of the form 
$[2^i j, 2^i(j+1)]$
where $i, j$ are nonnegative integers and $2^i(j+1)<2^s$.

\begin{lem}\label{lem;effective uniform}
One has 
\begin{align}
\sum_{[b, c]\in L_s} 
\int_Y 
\left(
\int_b^c{F(x,t)}\dd t
\right)^2 
 \dd\nu(x)\le 4C \delta^{-1} s2^s.
\end{align}
\end{lem}

\begin{proof} We estimate the left hand side as
\begin{align}
&\sum_{[b, c]\in L_s} 
\int_Y 
\left(
\int_b^c{F(x,t)}\dd t
\right)^2 
 \dd\nu(x) \notag\\
 \le & \sum_{i=0}^{s-1} \sum_{j=0}^{2^{s-i}-1}
 \int_Y
\left( \int_{2^ij}^{2^i(j+1)}{F(x,t)}\dd t \right)^2 \dd\nu(x)  \notag\\
\le & \sum_{i=0}^{s-1} \sum_{j=0}^{2^{s-i}-1}   {4C\delta^{-1}2^i},
   && \mbox{by Lemma \ref{lem;effective help}} \notag\end{align}
   which is clearly bounded from above by $4C \delta^{-1} s2^s$.
\end{proof}

\begin{lem}\label{lem;effective cover}
Let $k, s$ be positive integers with $ k < 2^s$. Then the interval 
$[0, k]$ can be covered by at most $s$ intervals in $L_s$.
\end{lem}

{\begin{proof} These intervals can be easily constructed using binary
    expansion of $k$. See also \cite[Lemma 1]{s60}. \end{proof}}

\begin{lem}\label{lem;effective key}
For every $\varepsilon >0$, there exists a sequence of measurable
subsets $\{Y_s\}_{s\in \N}$ of $Y$ such that: 
\begin{itemize}
\item[\rm (i)] $\nu(Y_s)\le 4C\delta ^{-1} s^{-(1+{2}\varepsilon)}$.
\item[\rm (ii)] For every positive  integer $k$ with $k<2^s$ and every ${y}\not \in Y_s$ one has 
\begin{align}\label{eq;effective k}
\left |\int_0^k {F(y,t)}\dd t\right |\le 2^{s/2} s^{{\frac32+\varepsilon} }. 
\end{align}
\end{itemize}
\end{lem}

\begin{proof}
Let 
\[
Y_s=
\left \{
{y}\in Y:\sum_{I\in L_s }    \left( \int_I {F(y,t)}\dd t 
\right )^2  > 2^s s^{2+{2}\varepsilon}
\right \}. 
\]
Assertion (i) follows from Lemma \ref{lem;effective uniform}
and 
{Markov's} Inequality. 
By  Lemma \ref{lem;effective cover} there exists a 
subset $L(k)$ of $L_s$ with cardinality at most $s$ such that 
$[0, k]=\bigcup_{I\in L(k)}I$.  For $k<2^s$ and  ${y}\not\in Y_s$ we {estimate}
\begin{align}
\notag
& \left(
\int_0^k {F(y,t)}\dd t
\right)^2      \\
\notag
\le & \left(\sum_{I \in L(k)} \int_I {F(y,t)} \dd t \right)^2
\\
\notag
 \le &\ s   \sum_{I\in L(k)}  \left( \int_I {F(y,t)}\dd t 
\right)^2     && \mbox{by Cauchy's inequality} \\
\notag
\le &\ s \sum_{I\in L_s }    \left( \int_I {F(y,t)}\dd t 
\right )^2    
\le 
\ 2^s s^{3+{2}\varepsilon}   && \mbox{since } y\not\in Y_s.\label{eq;effective line}
\end{align}
Now (\ref{eq;effective k}) follows 
by taking square roots. 
\end{proof}

\begin{proof}[Proof of Theorem \ref{thm;effective}]
We fix $\varepsilon>0$ and choose a sequence of measurable subsets 
$\{Y_s\}_{s\in \N}$ as in Lemma \ref{lem;effective key}. 
Note that 
\[
\sum_{s=1}^\infty\nu(Y_s)\le \sum_{s=1}^\infty4C\delta ^{-1}s^{-(1+{2}\varepsilon)}<\infty.
\]
The Borel-Cantelli lemma implies that there exists a measurable subset ${Y(\varepsilon)}$ of $Y$
with full measure such that for every $y\in {Y(\varepsilon)}$ there exists {$s_y\in\N$ such that $y\not\in Y_s$
whenever $s\ge  s_y$}.  

We will show that  
for 
every $y\in {Y(\varepsilon)}$ {one has
\eq{ll}{
 \frac{1}{T}\left|\int_0^T {F(y,t)}\dd t\right| {\ll}\ 
 T^{-{1/2}} \log^{{\frac32+\varepsilon} }(T)
}
 provided 
  $T$ is large enough, where the implicit constant depends only on ${F}$.}
Given
$T > 2$,
 let {$k = \lfloor T\rfloor$ and $s = 1 + \lfloor \log T\rfloor$, so that
 $2^{s-1}\le k \le T<k+1\le  2^s$.} 
 Suppose $T\ge 2^{s_y{-1}}$, then $s\ge s_y$ and hence $y\not\in Y_s$. 
Therefore we have
 \begin{align*}
 \left|\int_0^T {F(y,t)}\dd t\right| 
 \le &\  {\|F\|_\infty} +\left|\int_0^k {F(y,t)}\dd t\right| \\
 \le &\   {\|F\|_\infty} +2^{s/2}s^{{\frac32+\varepsilon} }
 && \mbox{by (\ref{eq;effective k})} \\
 \le&\  {\|F\|_\infty} +(2T)^{{1/2}}\log^{{\frac32+\varepsilon} }(2T).
  \end{align*}
{ and \equ{ll} follows.
This clearly implies (\ref{eq;effective goal})  for $y\in \cap_{k\in\N}Y(1/k)$.} 
 \end{proof}

\begin{proof}[Proof of Theorem \ref{thm;main rate}] Recall that we are
  given $\Lambda \in \Xn$, $\varphi\in C_c^\infty(\Xn)$ and
  $\varepsilon>0$. 
Take $f\in C_c^\infty(M)$ with $f\ge 0$ and $\int_M f\dd \vartheta=1$.  
Let $\nu$ be the probability measure on $\Xn$ defined by 
\[
\int_{\Xn}\psi \dd \nu=\int_M f(\vartheta)\psi\big(u(\vartheta)\Lambda\big )\dd \vartheta
\]
for every $\psi\in C_c({X})$. 
{Denote $\alpha=\int_{\Xn} \varphi \dd \mu$ and for $t,w\ge 0$ write
$$
\begin{aligned}
\int_{\Xn}\big(\varphi(g_t \Lambda )-{\alpha}\big)&\big(\varphi(g_{w}
\Lambda )-{\alpha}\big) \dd\nu =
I_{\Lambda,\Lambda,f,\varphi,\varphi}( t,{w}) -  
\int_M f \dd\vartheta \left( \int_{\Xn} \varphi\dd\mu\right)^2\\ 
&-\ 
\alpha \left(\int_{\Xn} f(\vartheta)\varphi\big(g_t u(\vartheta)\Lambda\big )\dd \vartheta \ - 
\int_M f \dd\vartheta \int_{\Xn} \varphi\dd\mu\right) \qquad  \\
&-\ 
\alpha \left(\int_{\Xn} f(\vartheta)\varphi\big(g_w u(\vartheta)\Lambda\big )\dd \vartheta - 
\int_M f \dd\vartheta \int_{\Xn} \varphi\dd\mu\right). \qquad\ 
\end{aligned}
$$
Applying  Theorem \ref{thm;mixing}  and Proposition \ref{prop;one mixing},
we conclude  that there exist $C>0$ and $\delta>0$ such that the estimate 
$$
\left|\int_{\Xn}\big(\varphi(g_t \Lambda )-{\alpha}\big)
  \big(\varphi(g_{w} \Lambda )-{\alpha}\big) \dd\nu\right|\le \  C
e^{-\delta\min(t,  {w}-t)} 
$$
holds  for any ${w} \geq
  t \geq 0$.
{Then we can apply} Theorem \ref{thm;effective}  with $F(x,t) = \varphi(g_tx) - \alpha$} 
{and obtain} (\ref{eq;main rate goal}) for 
almost every $\vartheta\in S_f$, where $S_f:= \{\vartheta\in M: f(\vartheta)>0\}$.
Since countably many {sets of the form} $S_f$ 
exhaust $M$, we reach the desired conclusion.
\end{proof}

\begin{rem} \label{alternative proof} Arguing similarly with $\nu$
    replaced by $\mu$  and using  exponential mixing of the
    $g_t$-action, see e.g.\ \cite[Theorem 1.1]{km12}, instead of
    Theorem \ref{thm;mixing},  
one can easily obtain \equ{eq;main rate in X}  for 
almost every $\Lambda\in\Xn$. \end{rem}

\section{Lattice points counting}\label{sec;count}
In this section we  recall a result of Schmidt \cite{s60} concerning a
counting problem arising from Diophantine approximation,  and relate it to the 
  $D$-action on $\Xn$.  From this we will deduce some estimates which
  will be used in the proof of 
Theorem \ref{thm;main}.

Let  $\mathbf{a}, \mathbf{b}$ be weight vectors, let
$\pi_{\mathbf{a}}, \pi_{\mathbf{b}}$ be 
 as in \equ{defpi}, 
and let $c, A, B$ be as in Theorem
\ref{thm;diophantine}. 
For an individual vector $\bv = (\mathbf{x}, \mathbf{y}) \in \R^d$ we
write $g_t\bv = (\mathbf{x}_t, \mathbf{y}_t)$, and then we have
$$\|\mathbf{y}_t\|_{\mathbf{b}} = e^{-t} \|\mathbf{y}
\|_{\mathbf{b}}$$
and
$$\|\mathbf{x}_t\|_{\mathbf{a}} \cdot \| \mathbf{y}_t \|_{\mathbf{b}} =
\|\mathbf{x}\|_{\mathbf{a}} \cdot \| \mathbf{y}\|_{\mathbf{b}}, \
\pi_{\mathbf{a}}(\mathbf{x}_t) = \pi_{\mathbf{a}}(\mathbf{x}), \ 
\pi_{\mathbf{b}}(\mathbf{y}_t) = \pi_{\mathbf{b}}(\mathbf{y})$$ 
for all $t$. 

For $r>0$, define \eq{deff}{f_{A,B,r,c}:=\mathbbm 1_{E_{r,c}(A,B)}}  to be the characteristic function of  $E_{r,c}(A,B)$.  
It follows that 
$$
f_{A,B,r,c}(g_t\bv) = 
\begin{cases} 1 & 
\begin{aligned}
 \text{if }
&e^t \leq \|\mathbf{y}\|_{\mathbf{b}} < e^{t+r}, \|\bx\|_\ba\cdot \|\by\|_\bb<c
\\
&\text{ and }
(\pi_{\mathbf{a}}\big(\mathbf{x}), \pi_{\mathbf{b}}(\mathbf{y})\big) \in A \times B
\end{aligned}
\\
0 & \text{otherwise.}
\end{cases} 
$$
Therefore 
$$
\bv \in E_{T,c}(A,B) \ssm E_{r,c}(A,B) \implies |\{ t \in [0,T]: g_t\bv \in
E_{r,c}(A,B) \}| = r
$$
and
$$
g_t\bv \in E_{r,c}(A,B) \text{ for some } t \in [0,T] \implies \bv \in E_{r+T,c}(A,B).
$$
Using \equ{eq: Siegel transform}, and changing the order
of summation and integration, it follows that for any
$\Lambda \in \Xn$ and any $T>r$ we have 
\eq{eq;key link}
{\begin{aligned}
\sharp\Big (\Lambda \cap \big(E_{T,c}(A,B)\ssm E_{r,c}(A,B)\big)\Big)
&\le \frac{1}{r}\int_0^T\widehat{
f}_{A,B,r,c}(g_t \Lambda)  \dd t \\ &\le
\sharp \big( E_{ r+T,c}(A,B)\cap \Lambda\big).
\end{aligned}}

Let $\mathbb S^{{\m}-1}_+:=\{\bx\in \mathbb S^{{\m}-1}: x_i\ge 0\}$ and
$\mathbb S^{{\n}-1}_+:=\{\by\in \mathbb S^{{\n}-1}: y_j\ge 0\}$. Also for brevity let us denote by $\Lambda_\vartheta$ the lattice $u(\vartheta)\Lambda_0$.
{Our main tool for proving Theorem \ref{thm;main}} is the following result 
of Schmidt.
\begin{thm}[\cite{s60}{,} Theorem{s} 1 and 2]\label{thm;schmidt use}
For almost every $\vartheta\in M$, 
\begin{align}\label{eq;count s60}
\sharp \left(E_{T, c}(\mathbb S^{{\m}-1}_+,\mathbb S^{{\n}-1}_+)\cap \Lambda_{\vartheta}
\right)\thicksim  cT \qquad \mbox{as }T\to \infty.
\end{align}
\end{thm}
The case ${\bb}=\bn$ of this result follows from \cite[Thm.\ 1 and {2} ]{s60} 
(setting $h=e^T, \, \psi_i(q) =
\big(\frac{c}{q}\big)^{a_i} $ in Schmidt's notation),
and in fact Schmidt also obtains an
error estimate $o\left(T^{1/2} \log^{\tau+\varepsilon} (T) \right)$ where $\tau=2$
if $n=1$ and $\tau=1.5$ otherwise. The
case ${\bb}\neq \bn$ does not appear in \cite{s60}, 
but the argument given there works for arbitrary $\bb$.

\begin{cor}\label{cor;count equi}
Let $r, c>0$.
Then for almost every $\vartheta\in M$, 
\begin{align}\label{eq;count equi}
\lim_{T\to \infty}\frac{1}{T}\int_0^T \widehat{
f}_{\mathbb S^{{\m}-1},\mathbb S^{{\n}-1},r,c}
(g_t \Lambda_{\vartheta} )\dd t =|E_{r,c}
|.
\end{align}
\end{cor} 
\begin{proof}
Let $N_{T,c}( \vartheta)$
be 
the number of 
solutions
 $(\mathbf p, \mathbf q)\in \mathbb Z^{{\m}}\times \mathbb Z^{\n}_{\ge 0}$ of the system 
\begin{align*}
\begin{cases}
& 0\le (\vartheta\mathbf q)_i-p_i< c^{a_i}\|\mathbf q\|_{\mathbf b}^{-a_i} \qquad (i=1, \ldots, {\m}) 
\\
&1\le \|\mathbf q\|_{\mathbf b}< e^T.
\end{cases}
\end{align*}
It follows directly from the definitions that 
\begin{align}\label{eq;count equi n}
N_{T,c}(\vartheta)=\sharp\left (E_{T, c}(\mathbb S^{{\m}-1}_+, \mathbb S^{{\n}-1}_+)\cap \Lambda_{\vartheta}\right).
\end{align}

Let $\{\mathbf e_i:1\le i\le {\m}\}$ be the standard basis of $\R^{\m}$.
 For $I\subset \{1, \ldots, {\m}\}$ 
 we let $\zeta_I : \R^{\m} \to \R^{\m}$ be the linear transformation defined
 by 
$$\zeta_I\mathbf e_i= \begin{cases} -\mathbf e_i & i\in I \\ 
\mathbf e_i & i \notin I. \end{cases}$$ 
Analogously we define $\eta_J: \R^{\n} \to
 \R^{\n}$  for $J\subset \{1, \ldots, {\n}\}$. 
It follows from Theorem \ref{thm;schmidt use} that for almost every $\vartheta\in M$
\begin{align}\label{eq;count number}
N_{T,c}( \zeta_I\vartheta \eta_J)\thicksim c  T \qquad \mbox{as } T \to \infty.  
\end{align}
On the other hand, it is easy to see that 
$
N_{T,c}( \zeta_I\vartheta \eta_J)
$
is the number of 
solutions
$(\mathbf p, \mathbf q)\in \mathbb Z^{ d }$ of the system 
\begin{align*}
&c^{a_i}\|\mathbf q\|_{\mathbf b}^{-a_i}<  (\vartheta\mathbf q)_i-p_i\le 0 &(i \in I)\notag\\
& 0\le (\vartheta\mathbf q)_i-p_i< c^{a_i}\|\mathbf q\|_{\mathbf b}^{-a_i} &(i\not \in I) \notag\\
&1\le \|\mathbf q\|_{\mathbf b}< e^T  & \notag \\
& q_j\le 0 &(j \in J)  \notag\\
& q_j\ge 0 &(j\not \in J) \notag
\end{align*}

Now let $\widetilde N_{T,c}( \vartheta)$ be the number of 
solutions $(\mathbf p, \mathbf q )\in \mathbb Z^{ d }$
of the system \equ{kg1}--\equ{kg2}.
Then  for almost every $\vartheta\in M$ (i.e.\ for those $\vartheta$ for
which (\ref{eq;count number}) holds for all $I$ 
and 
$(\vartheta\mathbf q)_i-p_i$ is not equal to zero for any $\mathbf{p}, \mathbf{q}$ with $\bq\neq 0$ and any $i$), one has 
\begin{align}\label{eq;diophantine claim}
\widetilde N_{T,c}( \vartheta)\thicksim  2^{ d } cT\qquad\mbox{as }T\to \infty.
\end{align}
As was mentioned in the introduction, 
\begin{align}\label{eq;count equi erc}
\widetilde N_{T,c}(\vartheta)=\sharp\left (E_{T, c}\cap \Lambda_{\vartheta}\right)\text{ \  and  \ }|E_{T, c}|=2^dcT .
\end{align}
It follows from 
(\ref{eq;count equi erc})
and (\ref{eq;key link})
that  (\ref{eq;count equi}) holds for those $\vartheta$ which satisfy
(\ref{eq;diophantine claim}).
\end{proof}

\begin{cor}\label{cor;count equi f}
For $r, c>0$ we let
\begin{align}\label{eq;f r c}
F_{r, c}:=
\left\{(\mathbf x, \mathbf y)\in \mathbb R^d: \|\bx\|_{\mathbf a}\|\mathbf
y\|_{\mathbf b}< c, \, 1\le  
\|\mathbf x\|_{\mathbf a}<e^r \right\}.
\end{align}
Then for almost every $\vartheta\in M$
\begin{align*}
\lim_{T\to \infty}\frac{1}{T}\int_0^T \widehat{\mathbbm 1} _{F_{r,c}}(g_t \Lambda_{\vartheta})\dd t =|F_{r,c}|.
\end{align*}
\end{cor}
\begin{proof}
Let 
$\widetilde {N}_{T, c}$ be as in Corollary \ref{cor;count equi}. 
For every  $T> |\log c|$   and $\vartheta\in M$, similarly to (\ref{eq;key link}) one has
\[
\frac{1}{r}\int_0^T \widehat{\mathbbm 1}_{F_{r,c}} (g_t\Lambda_{\vartheta})\, dt\le
  \widetilde{ N}_{T+\log c}(\vartheta)+\sharp (\widetilde F\cap \Lambda_{\vartheta})
\]
where 
\[
\widetilde F:=\{(\mathbf x, \mathbf y)\in \mathbb R^d: 
\|\mathbf x\|_{\mathbf a}<e^r, 
\|\mathbf y\|_{\mathbf b}<c\}.
\]
For every $\vartheta$, $\sharp (\widetilde F\cap \Lambda_{\vartheta})$
is a number independent of $T$ and $|F_{r,c}|=|E_{r,c}|$. So for  $\vartheta\in M$
satisfying (\ref{eq;diophantine claim}) one has
\[
\limsup_{T\to \infty}
\frac{1}{T}\int_0^T \widehat{\mathbbm 1}_{F_{r,c}}
(g_t\Lambda_{\vartheta})\dd t\le |F_{r, c}|.
\]
On the other hand by \cite[Corollary 1.3]{s}, there is a conull subset
of $\vartheta \in M$ for which $\Lambda_{\vartheta}$ is $\big(D^+,
C_c(\Xn)\big)$-generic. For any $\vre>0$, since $\mathbbm{1}_{F_{r,c}} $
is a   {Riemann integrable} non-negative function, there is $\varphi_1 \in
C_c(\R^d)$ such that 
$$
\mathbbm{1}_{F_{r,c}}(\bv) \geq \varphi_1(\bv) \text{ for all } \bv \text{ and
} \int_{\R^d} \varphi_1 > |F_{r,c}|-\vre.
$$
Now $\widehat{\varphi}_1$ is a continuous non-negative 
integrable function on $\Xn$, and therefore there exists $\varphi_2 \in
C_c(\Xn)$ such that 
\eq{eq: second inequality1}{
\varphi_2(\Lambda) \leq \widehat{\varphi}_1(\Lambda) \leq \widehat{\mathbbm{1}}_{F_{r,c}}(\Lambda)
\text{ for all } \Lambda
}
and 
\eq{eq: second inequality}{
\int_{\Xn} \varphi_2 \dd \mu \geq \int_{\Xn} \widehat{\varphi}_1 \dd
\mu -\vre  = \int_{\R^d} \varphi_1 - \vre \geq |F_{r,c}| -2\vre.
}
Since $\vre>0$ was arbitrary, \equ{eq: second inequality1} and
\equ{eq: second inequality} imply that for almost every $\vartheta\in M$
\[
\liminf_{T\to \infty}
\frac{1}{T}\int_0^T \widehat{\mathbbm 1}_{F_{r,c}}
(g_t\Lambda_{\vartheta})\dd t\ge  |F_{r, c}|.
\]
This completes the proof. 
\end{proof}

\section{Pointwise equidistribution with respect to unbounded functions}\label{sec:unbdd}
In this section we prove 
Theorems \ref{thm;main} and \ref{thm;diophantine}. We first show that
for  Riemann integrable $f,$   {the function} $\widehat{f}$   {as in \equ{eq: Siegel transform}} belongs to the class $C_\alpha(X)$.

\begin{lem}\label{lem: Blichfeldt}
For any $d$ and all sufficiently large $r$ there are constants $c_1,
c_2$ such that 
if $\mathbbm 1_{B_r}$  is the characteristic function of the open ball 
of radius $r$ centered at origin, 
 then for all $\Lambda \in X $, 
\begin{align}\label{eq;upper lower}
c_1 \alpha(\Lambda) \leq \widehat{\mathbbm{1}}_{B_r}(\Lambda)
\leq c_2 \alpha(\Lambda).
\end{align}
In particular, for any Riemann   {integrable} $f: \R^d \to \R$,
$\widehat{f} \in C_\alpha(X)$. 
\end{lem}
\begin{proof}
For any discrete subgroup $  {\Delta} \subset \R^d$, the $i$-th
Minkowski 
successive minimum of $  {\Delta}$ with respect to $B_r$ is defined to be 
$$
\lambda_i\left(  {\Delta}\right) := \inf\left\{t>0: \dim \big( \spa\,
 (tB_r \cap   {\Delta})\big) \geq i \right\}.
$$
Let $r_0$ be large enough (depending on $d$) so that for any $r \geq
r_0$ and any $\Lambda
\in X$, $\lambda_1( \Lambda) <1$. 
The notation $x\asymp y$ will mean that $x$ and
$y$ are functions of discrete subgroups of  $\R^d$ and there are positive
constants $C_1, C_2$, depending on $d$ and $r$, such that
$\displaystyle{C_1 \leq x/y \leq C_2}$.  By Minkowski's second theorem, see e.g.\ 
\cite[\S VIII.2]{ca}, 
$$\lambda_1\left(  {\Delta} \right) \cdots \lambda_\ell\left(  {\Delta}\right)\asymp d\left(  {\Delta} \right),$$ where $d
\left(  {\Delta} \right)$
is the covolume of $  {\Delta}$ in $\spa \left(  {\Delta} \right)$ and $\ell$ is the rank
of $  {\Delta}$.   {Now for $\Lambda
\in X$ define $\Delta $ to be the subgroup of
$\Lambda$ generated by $\Lambda \cap \overline{B_r}$. Then}
$$\alpha(\Lambda) \asymp \big(\lambda_1(\Lambda) 
\cdots \lambda_j(\Lambda)\big)^{-1} =
\big(\lambda_1(  {\Delta}) 
\cdots \lambda_j(  {\Delta})\big)^{-1},$$ where $j$ is   {the} index for which $\lambda_j(\Lambda) \leq 1
< \lambda_{j+1}(\Lambda)$.

It follows from
\cite[Prop.\ 2.1 and Cor.\ 2.1]{Henk} (applied to $K=B_r, \, d=j$ and $\mathbb{L}=  {\Delta}$) that 
$$
  {\sharp}(  {\Delta}\cap B_r)\asymp
\big(\lambda_1 (  {\Delta} )\cdots \lambda_{j}(  {\Delta})\big)^{-1}.
$$
Since $\widehat{ \mathbbm{1}}_{B_r} (\Lambda) =  {\sharp}(\Lambda \cap B_r)   {-1}=
  {\sharp}(  {\Delta}\cap B_r)  {-1}$, \equ{eq;upper lower} follows. 

For the second assertion, note that for any Riemann   {integrable} $f:\R^d \to \R$ there are positive $r$ and $C$ so that $|f| \leq
C\cdot \mathbbm{1}_{B_r}$ and hence condition ($C_\alpha$-$2$) follows from
\equ{eq;upper lower}. To prove ($C_\alpha$-$1$), let $S$ be the set of
discontinuities of $f$ in $\R^d \sm \{0\}$, so that $|S|=0$. From
\equ{eq: Siegel transform} it follows that the set
$S'$ of discontinuities of
$\widehat{f}$ is contained in $$S'':=\{\Lambda: \Lambda \cap S \neq
\varnothing\}.$$ For each $\bv \in \Z^d \sm \{0\}$, the set of
$g \in G$ such that $g\bv \in S$ has Haar measure zero in $G$, and hence
$S''$ is a countable union of sets of $\mu$-measure zero. In
particular $\mu(S')=0$. 
\end{proof}

We now derive some general properties  from  equidistribution of measures.

\begin{lem}\label{lem;truncation}
Let $\psi \in C_\alpha(X)$ and
let $\{\mu_i\}$ be a sequence of probability measures on $\Xn$
such that $\mu_i \to \mu$, 
with respect to the weak-$*$ topology. 
Then for any non-negative $\varphi \in C_c( \Xn)$ one has
\begin{equation}\label{eq;compact}
\lim_{i\to \infty}\int_{\Xn}\varphi
\psi\dd\mu_i =\int_{\Xn} \varphi
\psi \dd\mu.
\end{equation}
\end{lem}
\begin{proof} 
Using assumption $(C_\alpha$-$1)$ we see that the function $\varphi\psi$ is bounded,  compactly 
supported and continuous except on a set of measure zero. {By using a partition of unity, without loss of generality one can assume that $\varphi$ is supported on a coordinate chart. 
Applying Lebesgue's criterion for Riemann integrability to $\varphi\psi$, one can write $\int_X \varphi\psi\dd\mu$ as the limit of  upper and lower Riemann sums.  From this it  easily follows that}
 for any $\varepsilon>0$ there exist $h_1,h_2\in C_c(X)$  
such that $h_1\le \varphi\psi\le h_2$ and 
\begin{align}
\int_{X}(h_2-h_1)\dd\mu \leq\varepsilon.
 \label{eq;truncation 4}
\end{align}
Thus we have
\begin{align}
&\limsup_{ i\to \infty}\int_{\Xn}\varphi
\psi\dd\mu_i \le \int_{\Xn} 
h_2 \dd\mu
 \label{eq;truncation 1}\\
&\liminf_{ i\to \infty}\int_{\Xn}\varphi
\psi\dd\mu_i \ge \int_{\Xn} 
h_1 \dd\mu 
 \label{eq;truncation 2}\\
& \int_{\Xn}
h_1\dd\mu\le \int_{\Xn} \varphi
\psi\dd\mu 
\le \int_{\Xn}
h_2 \dd\mu.  \label{eq;truncation 3}
\end{align}
Therefore by taking $\varepsilon\to 0$, (\ref{eq;compact}) follows from 
 (\ref{eq;truncation 1})--(\ref{eq;truncation 3})
 and (\ref{eq;truncation 4}).
\end{proof}

\begin{cor}\label{cor;tail}
Let the notation be as in Lemma \ref{lem;truncation} and assume that
\begin{align}\label{eq;equi siegel}
\lim_{i\to \infty}\int_{\Xn} 
\psi\dd\mu_i=\int_{\Xn} 
\psi\dd\mu.
\end{align}
Then for any $\varepsilon >0$ there exists $i_0>0$ and
$\varphi \in C_c(\Xn)$ with $0 \leq \varphi \leq 1$ such that
\begin{align}\label{eq;con tail}
\left|\int_{\Xn} (1-\varphi)
\psi \dd\mu_i\right |< \varepsilon
\end{align}
for any $i\ge i_0$.
\end{cor}
\begin{proof}
Since $
\psi \in L^1(\Xn,\mu)$,
there exists a compactly supported continuous function 
 $\varphi: \Xn\to [0, 1]$ such that 
 \begin{align}\label{eq;tail 1}
\left|\int_{\Xn} (1-\varphi)
\psi\dd\mu\right|<\frac{\varepsilon}{3}.
 \end{align}
 By Lemma \ref{lem;truncation} and (\ref{eq;equi siegel}), there
 exists $i_0>0$ such that for $i\ge i_0$ 
 \begin{align}
\left | \int_{\Xn} \varphi
\psi
\dd\mu_i- \int_{\Xn} \varphi
\psi
\dd\mu\right|&<\frac{\varepsilon}{3}
\label{eq;tail 2}\\
\left | \int_{\Xn}
\psi
\dd\mu_i- \int_{\Xn}
\psi
\dd\mu\right|&<\frac{\varepsilon}{3}
\label{eq;tail 3}.
 \end{align}
 It follows from (\ref{eq;tail 1}),
 (\ref{eq;tail 2}) and (\ref{eq;tail 3}) that  (\ref{eq;con tail}) holds if $i\ge i_0$.
\end{proof}

\begin{cor}\label{cor;dominated}
Let the notation be as in Lemma \ref{lem;truncation}.
 Assume  that
there exists a non-negative Riemann 
integrable  function 
$f_0$ on $\mathbb R^{ d }$ such that  $|\psi |\le \widehat{f_0}$ and 
\begin{align}\label{eq;dominated assume}
\lim_{i\to \infty}\int_{\Xn} \widehat {f_0} \dd\mu_i=\int_{\Xn}
\widehat {f_0} \dd\mu. 
\end{align}
Then
\[
\lim_{i\to \infty}\int_{\Xn} \psi\dd\mu_i=\int_{\Xn}  \psi \dd\mu.
\]
\end{cor}
\begin{proof}
By Lemma \ref{lem;truncation}, Corollary \ref{cor;tail} and
(\ref{eq;dominated assume}), 
for any $\varepsilon>0$ there exists $i_0>0$
and a continuous compactly supported function $\varphi: \Xn\to [0, 1]$
such that for $i\ge i_0$ one has
\begin{align*}
\left| \int_{\Xn}\varphi \psi \dd\mu_i-\int_{\Xn}\varphi \psi \dd\mu\right|
&<\frac{ \varepsilon}{3}, \\
\int_{\Xn}(1-\varphi) {\widehat f_0 }\dd\mu_i&< \frac{\varepsilon}{3}, \\
\int_{\Xn}(1-\varphi) {\widehat f_0 }\dd\mu &<\frac{ \varepsilon}{3}.
\end{align*}
Using the assumption $|\psi \, |\le \widehat{f_0}$ and these
inequalities, one finds
\[
\left|\int_{\Xn} \psi \dd\mu_i-\int_{\Xn}  \psi \dd\mu\right|<\varepsilon
\]
for $i\ge i_0$.
This completes the proof.
\end{proof}

Now, given $\psi\in C_\alpha(X)$, we are going to apply the results of
\S\ref{sec;count} to construct a function $f_0$  satisfying the
assumption of Corollary \ref{cor;dominated}.

\begin{lem}\label{lem;diophantine}
For $r> d $ let  $A_r$ be the annular region defined by 
$$A_r:=\{\mathbf v\in \mathbb R^{ d }:  d <\|\mathbf v\|<r\}.$$
Then for almost every $\vartheta\in M$, 
\begin{align}\label{eq;annular}
\lim_{T\to \infty}\frac{1}{T}\int_0^T 
\widehat {\mathbbm 1}_{A_r}\big(g_t
\Lambda_\vartheta\big)\dd t=|A_r|.
\end{align}
\end{lem}

\begin{proof}
For ${r'},c>0$ let
$F_{{r'}, c}$ be as in (\ref{eq;f r c}).
Since  $r>  d $, there exist $c,{r'}> 1$ such that 
$A_r\subset E_{{r'}, c}\cup F_{{r'}, c}$.  It follows that 
\[
{\mathbbm 1}_{A_{r}}\le 
{\mathbbm 1}_{E_{{r'},c}}+{\mathbbm 1}_{F_{{r'},c}}.
\]
It follows from Corollaries \ref{cor;count equi} and \ref{cor;count equi f} that (\ref{eq;dominated assume}) holds
for $$f_0={\mathbbm 1}_{E_{{r'},c}}+{\mathbbm 1}_{F_{{r'},c}}.$$
Hence the conclusion of the lemma follows  from Corollary \ref{cor;dominated}.
\end{proof}

\begin{lem}\label{lem;key}
Let $B_r$ be the open Euclidean ball of radius $r>0$ centered at the origin in $\mathbb R^{ d }$. 
Then 
\begin{align*}
\lim_{T\to \infty}\frac{1}{T}\int_0^T \widehat {\mathbbm{1}}_{B_r}(g_t
\Lambda_{\vartheta})\dd t=
|B_r|
\end{align*}
for almost every $\vartheta\in M$. 
\end{lem}
\begin{proof}
We first observe that any annular region of width $d$ centered at the
origin contains a lattice point of  
 any unimodular lattice in  $ \mathbb R^{ d }$. 
 Also for any unimodular lattice $\Lambda$ in $\R^d$ and any $\bv\in
 \Lambda$ one has 
 $$\sharp (B_r\cap \Lambda)=\sharp \big((B_r+\bv)\cap \Lambda\big)\,.$$
 It follows from these two observations that for any lattice $\Lambda
 \in X$, the number of lattice points in $B_r$ is at least the number
 of lattice points in $A_{r'}$, where $r' = 2d+2r$. That is, 
 $
\widehat{  {\mathbbm{1}}}_{B_r}\le  \widehat{\mathbbm{1}}_{A_{r'}}.
 $
The conclusion now  follows from Lemma \ref{lem;diophantine} and Corollary \ref{cor;dominated}.
 \end{proof}
 
 \begin{proof}[Proof of Theorem \ref{thm;main}] 
 It follows from Lemma \ref{lem;key} and the Siegel integral formula 
 that there is a Borel  subset $M'$ of $M$ 
  of full measure such that for 
 any $\vartheta\in M'$ 
 and  $r\in \mathbb \N$  one has
 \[
 \lim_{T\to \infty}\frac{1}{T} \int_0^T \widehat{\mathbbm 1}_{B_r}(g_t \Lambda_\vartheta)\dd t
 =\int_X \widehat{\mathbbm 1}_{B_r}\dd\mu.
 \]
 Moreover we assume that $\Lambda_\vartheta$ is $\big(D^+, C_c(X)\big)$-generic for  every $\vartheta\in M'$. 
 In view of (\ref{eq;upper lower}) for any $\psi\in 
 C_\alpha(X)$ there is $c>0$ 
 and $r\in \N$ such that 
 \[
 |\psi(\Lambda)|\le c\widehat{\mathbbm 1}_{B_r}(\Lambda)
 \]
 for every $\Lambda \in X$. 
 It follows from Corollary \ref{cor;dominated} that 
 $\Lambda_\vartheta$ is $(D^+,\psi)$-generic for every $\vartheta\in M'$,
 and 
 the proof is complete.  
 \end{proof}

\begin{proof}[Proof of Theorem \ref{thm;diophantine}] Our argument is similar to the one used in \cite{apt}. 
The assumption on the boundaries of $A$ and $B$ 
in $\mathbb S^{{\m}-1}$ and $\mathbb S^{{\n}-1}$ respectively implies that the boundary of  
$E_{r, c}(A, B)$ in $\R^d$ has zero Lebesgue measure.   
With the notation $f_{A,B,r,c}$ as in \equ{deff}, it follows from
\equ{eq;key link} and Theorem \ref{thm;main} that for almost every
$\vartheta\in M$ 
\begin{equation}\begin{aligned}\label{eq;application}
\sharp\big(E_{T,c}(A, B) \cap \Lambda_{\vartheta}\big) \cdot r  & \thicksim
\int_0^T\widehat{f}_{A,B,r,c}
(g_t \Lambda_{\vartheta})\dd t\\ &\thicksim |E_{r,c}(A, B)|\cdot T \quad \mbox{as }T \to 
\infty.
\end{aligned}\end{equation}

In order to conclude the proof,
 it suffices to show that 
\begin{align}\label{eq;compare}
|E_{r,c}(A, B)|\cdot T =|E_{T,c}(A, B)| \cdot r.
\end{align}
Indeed, when $T = kr $ for some $k \in \N$ we have $$E_{T,c}(A, B) =
\bigcup_{j=0}^{k-1} g_{-jr} \big(E_{r,c}(A, B)\big)$$ (a disjoint union), and \equ{eq;compare} follows from the fact
that the $g_t$-action preserves Lebesgue measure. From this one
deduces \equ{eq;compare} when  $T/r\in\Q$, and finally one gets
\equ{eq;compare} for arbitrary $T,r>0$ by continuity. 
\end{proof}

\end{document}